\documentclass{amsart}
\usepackage{color, latexsym, graphicx, amsmath, amssymb, cite, hyperref, tikz-cd}
\usepackage[margin=1.5in]{geometry}

\newtheorem{theorem}{Theorem}[section]
\newtheorem{lemma}[theorem]{Lemma}
\newtheorem{proposition}[theorem]{Proposition}

\theoremstyle{definition}
\newtheorem{definition}[theorem]{Definition}
\newtheorem{example}[theorem]{Example}

\theoremstyle{remark}

\numberwithin{equation}{section}

\newcommand{\R}{{\mathbb R}}

\newcommand{\Z}{{\mathbb Z}}

\newcommand{\N}{{\mathbb N}}
\newcommand{\T}{{\mathbb T}}

\newcommand{\supp}{\operatorname{supp}}
\newcommand{\ord}{\operatorname{ord}}
\newcommand{\WF}{\operatorname{WF}}

\newcommand{\inj}{\operatorname{inj}}
\newcommand{\esssupp}{\operatorname{ess\, supp}}

\newcommand{\vol}{\operatorname{vol}}

\title{Polygons and multi-product of eigenfunctions}

\author{Emmett L. Wyman}
\address{Department of Mathematics and Statistics, Binghamton University, Binghamton NY}
\email{ewyman@binghamton.edu}

\author{Yakun Xi}
\address{School of Mathematical Sciences, Zhejiang University, Hangzhou 310027, China}
\email{yakunxi@zju.edu.cn}
\author{Yi Zhang}
\address{School of Mathematical Sciences, Zhejiang University, Hangzhou 310027, China}
\email{12335009@zju.edu.cn}

\begin{document}

\begin{abstract}
Let $M$ be a compact Riemannian manifold without boundary, with $L^2$-normalized Laplace-Beltrami eigenfunctions $\{e_j\}_j$, which satisfy $\Delta_g e_j = -\lambda_j^2 e_j$. We study the following inner product of eigenfunctions
\[
	\langle e_{i_1} e_{i_2} \ldots e_{i_k}, e_{i_{k+1}} \rangle = \int e_{i_1} e_{i_2}\ldots e_{i_k} \overline{e_{i_{k+1}}} \, dV.
\]
We show that, after a mild averaging in the frequency variables, the main $\ell^2$-concentration of this inner product is determined by the measure of a set of configurations of $(k+1)$-gons whose side lengths are the frequencies $\lambda_{i_1}, \lambda_{i_2}, \dots, \lambda_{i_{k+1}}$. We prove that a rapidly vanishing proportion of this mass lies in the regime where $\lambda_{i_1}, \lambda_{i_2}, \dots, \lambda_{i_{k+1}}$ cannot occur as the side lengths of any $(k+1)$-gon.

\end{abstract}
\maketitle

\section{Introduction} 

\subsection{Background}

Let $M$ be a compact Riemannian manifold without boundary, and $e_1, e_2, \ldots$ be an orthonormal basis of Laplace-Beltrami eigenfunctions with
\[
	\Delta_g e_j = -\lambda_j^2 e_j,
\]
where $\lambda_j$ is the frequency of $e_j$.

We are interested in the general behavior of the multi-product of $k$ eigenfunctions $e_{i_1} e_{i_2} \dots e_{i_k}$. To analyze this multi-product more conveniently, we express it as a sum of eigenfunctions with coefficients by a harmonic expansion
\begin{equation}\label{harmonic expansion}
	e_{i_1} e_{i_2} \dots e_{i_k} = \sum_{i_{k+1}} \langle e_{i_1} e_{i_2} \dots e_{i_k}, e_{i_{k+1}} \rangle e_{i_{k+1}}.
\end{equation}
The classical question is to study the decay of the $L^2$ norm of $e_{i_1} e_{i_2} \dots e_{i_k}$. If we set $\lambda_{i_1}=\lambda_{i_2}=\dots =\lambda_{i_k}=\lambda_j$, then the problem reduces to study the $L^{2k}$ norm of the eigenfunction $e_{\lambda_j}$. For this problem, Sogge provided a complete answer for $k\geq 1$, see \cite{MR930395}. 
However, for general eigenvalues $\lambda_{i_1}, \lambda_{i_2}, \ldots, \lambda_{i_k}$, the situation is more complicated. By orthogonality, we can write the $L^2$ norm of \eqref{harmonic expansion} as follows:
\begin{equation}\label{rewrite harmonic expansion}
	\|e_{i_1} e_{i_2} \cdots e_{i_k}\|_{L^2}^2 = \sum_{i_{k+1}\in\mathbb Z^+} \bigl|\langle e_{i_1} e_{i_2} \cdots e_{i_k}, e_{i_{k+1}} \rangle\bigr|^2.
\end{equation}
The decay of the terms on the right-hand side therefore determines the behavior of the left-hand side. We thus aim to study both the decay and how this decay is distributed in \eqref{rewrite harmonic expansion}. This leads to the following concrete problem: under what conditions on the frequencies $\lambda_{i_1}, \lambda_{i_2}, \dots, \lambda_{i_{k+1}}$ do the coefficients $\langle e_{i_1} e_{i_2} \cdots e_{i_k}, e_{i_{k+1}} \rangle$ exhibit decay, and how is their $\ell^2$-mass distributed as these frequencies vary?

When $k=2$, as in many spectral problems, this question is closely connected to number theory in the automorphic setting. Sarnak showed that for fixed $j$ the coefficients $\langle e_j^2, e_\ell \rangle$ decay exponentially in $\lambda_\ell$ \cite{MR1277052}. Bernstein, Reznikov, Kr\"otz, and Stanton obtained the optimal estimate \cite{MR1715328, MR2081437}. For the more general coefficients $\langle e_i e_j, e_\ell \rangle$, a classical decay estimate in terms of the eigenvalues was given by Lu, Sogge, and Steinerberger \cite{MR3997636}. More recently, the first author improved their estimate \cite{MR4376455}. Given $\epsilon > 0$ and an integer $N \geq 1$, there exists a constant $C_{\epsilon,N}$ for which
\begin{equation}\label{Emmett conclusion}
	\sum_{\lambda_\ell \geq (2+\epsilon)\lambda} |\langle e_i e_j, e_\ell \rangle|^2 \leq C_{\epsilon, N} \lambda^{-N} \qquad \text{for all $i$, $j$ with $\lambda_i, \lambda_j \leq \lambda$.}
\end{equation}

Furthermore, he proved that the $(2+\epsilon)\lambda$ factor is nearly optimal. He also obtained an asymptotic expansion for the non-rapidly decaying part \cite{MR4376455}. Building on this work, we study the multi-product of $(k+1)$ eigenfunctions and obtain analogous results. 
One of the most striking features of \cite{MR4376455} is that it connects triple products of eigenfunctions to geometric configurations, namely triangles. We extend this perspective to general $k \geq 2$, where the corresponding configurations are naturally described by polygons.

\subsection{Polygons}

Based on \eqref{rewrite harmonic expansion}, we study the sums
\begin{equation}\label{single sums}
	\sum_{(\lambda_{i_1}, \lambda_{i_2}, \dots ,\lambda_{i_{k+1}}) = (a_1, a_2, \dots ,a_{k+1})} \bigl|\langle e_{i_1} e_{i_2} \cdots e_{i_k}, e_{i_{k+1}} \rangle\bigr|^2
\end{equation}
for given eigenvalues $a_1, a_2, \dots, a_{k+1}$. These sums are independent of the choice of eigenbasis; hence their behavior is determined solely by the underlying manifold and its metric.

The main point of the present paper is to show that the sum \eqref{single sums} can be estimated by counting configurations of $(k+1)$-gons with side lengths $a_1, a_2, \dots, a_{k+1}$. To state our results, we first introduce the following definition concerning the side lengths of a $(k+1)$-gon.

\begin{definition}\label{k+1-good and k+1-gon-bad}
A vector $\tau = (\tau_1, \ldots, \tau_{k+1}) \in (0,\infty)^{k+1}$ is called $k$-good if it satisfies the following two conditions.
\begin{enumerate}
\item the $(k+1)$-gon inequality: for every $1 \leq j \leq k+1$,
\begin{equation}\label{2.1}
\tau_j < \sum_{\ell \neq j} \tau_{\ell}.
\end{equation}
\item the nondegeneracy condition: for every sign vector $\varepsilon = (\varepsilon_1, \ldots, \varepsilon_{k+1}) \in \{ \pm 1\}^{k+1}$,
\begin{equation}\label{2.2}
|\varepsilon \cdot \tau| = \left|\sum_{j=1}^{k+1} \varepsilon_j \tau_j\right| > 0.
\end{equation}
\end{enumerate}
We denote the set of $k$-good vectors by $\Gamma_g$.

A vector $\tau = (\tau_1, \ldots, \tau_{k+1}) \in (0,\infty)^{k+1}$ is called $k$-bad if there exists $j$, $1 \leq j \leq k+1$, such that
\begin{equation}\label{2.3}
\tau_j > \sum_{\ell \neq j} \tau_{\ell}.
\end{equation}
We denote the set of $k$-bad vectors by $\Gamma_b$.

A closed cone $\Gamma \subset (0,\infty)^{k+1}$ is called $k$-good if $\Gamma \subset \Gamma_g$, and it is called $k$-bad if $\Gamma \subset \Gamma_b$.
\end{definition}

The set of $\tau$ satisfying condition (1) consists of side-length data that can be realized by a $(k+1)$-gon in the plane, whereas the set of $k$-bad vectors contains no such points. The set of $k$-good vectors is obtained from the set of $\tau$ satisfying condition (1) by removing those $\tau$ for which a linear combination $|\varepsilon \cdot \tau|$ vanishes for some $\varepsilon \in \{ \pm 1\}^{k+1}$. For example, when $k=3$, the vector $(2,2,2,2)$ satisfies the $(k+1)$-gon inequality, but $(2,2,2,2)\notin \Gamma_g$ since it fails the nondegeneracy condition. Concretely, consider the degenerate quadrilateral in the plane formed by the vectors $(2,0), (-2,0), (2,0), (-2,0)$. Each side has length $2$, but the polygon is collinear and has zero area. In this case, $\tau \notin \Gamma_g$, as illustrated in Figure \ref{4}. When $k=2$, the nondegeneracy condition is automatic once one assumes condition (1), due to the rigidity of triangles. This is the primary reason it was not clear whether triple-product results could be generalized to higher multi-products, where such rigidity is absent. Our nondegeneracy condition is designed to fill this gap. As we will see in the proof, it ensures that several crucial steps go through and yields clean asymptotics.

\begin{figure}
\includegraphics[width=0.6\textwidth]{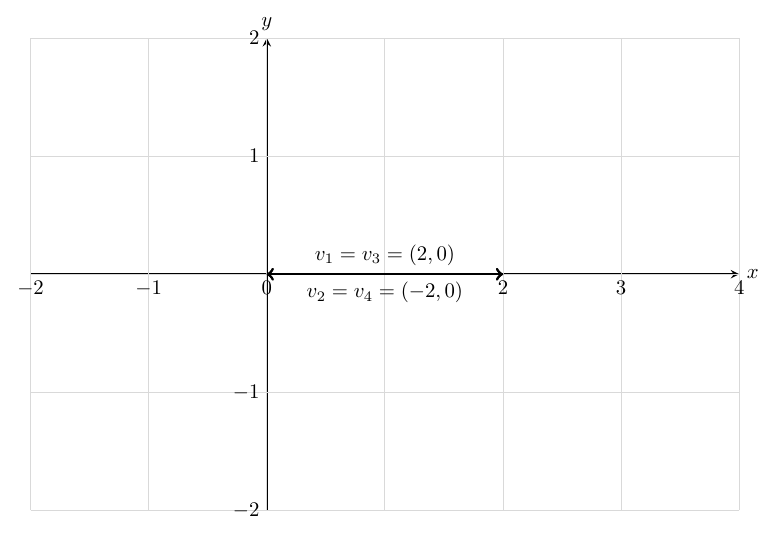}
\caption{An example of $\tau\notin\Gamma_g$.}
\label{4}
\end{figure}
In what follows, we establish a connection between \eqref{single sums} and the count of configurations of $(k+1)$-gons using the spectral projection operator $\chi_\lambda$ and the theory of Fourier integral operators. Our argument also shows that \eqref{single sums} decays rapidly in the regime where $(\lambda_{i_1}, \lambda_{i_2}, \dots, \lambda_{i_{k+1}})$ fails to satisfy the $(k+1)$-gon inequality.

\subsection{Main results}
Our first theorem shows that the rapidly decaying contribution to \eqref{single sums} comes from frequencies $(\lambda_{i_1}, \lambda_{i_2}, \dots, \lambda_{i_{k+1}})$ whose coordinates cannot occur as the side lengths of a $(k+1)$-gon.

\begin{theorem}\label{main theorem k+1-gon-bad}
	For each $\epsilon>0$ and integer $N > 0$, there exists a constant $C_{\epsilon, N}$ such that for any positive numbers $\tau_{1},\tau_{2},\ldots,\tau_{k}$, we have
	\[
		\sum_{\lambda_{i_{k+1}}\geq (1+\epsilon)(\tau_1+\tau_2+\dots+\tau_k)} \bigl|\langle e_{i_1} e_{i_2} \cdots e_{i_k}, e_{i_{k+1}} \rangle\bigr|^2 \leq C_{\epsilon, N}(\tau_1+\tau_2+\dots+\tau_k)^{-N},
	\]
	where $\lambda_{i_j} \leq \tau_j$ for $1 \leq j \leq k$.
\end{theorem}

The hypothesis $\lambda_{i_{k+1}} \geq (1+\epsilon)(\tau_1+\cdots+\tau_k)$ forces the $(k+1)$-tuple $(\tau_1,\ldots,\tau_k,\lambda_{i_{k+1}})$ to be $k$-bad, since the last entry dominates the sum of the others. In particular, any $(\lambda_{i_1},\ldots,\lambda_{i_k},\lambda_{i_{k+1}})$ with $\lambda_{i_j} \leq \tau_j$ for $1 \leq j \leq k$ is also $k$-bad.

For clarity, we do not organize the left-hand side by summing over all tuples $(\lambda_{i_1},\ldots,\lambda_{i_{k+1}})$ lying in a specified region of $(0,\infty)^{k+1}$. Instead, we fix upper bounds $\tau_j$ for $\lambda_{i_j}$, $1 \leq j \leq k$, and then estimate the tail sum over those $i_{k+1}$ with $\lambda_{i_{k+1}} \geq (1+\epsilon)(\tau_1+\cdots+\tau_k)$. The proof uses the method of non-stationary phase.
We record a flat-torus example in Section \ref{sec:flat-torus-example} showing that the frequency threshold in Theorem \ref{main theorem k+1-gon-bad} is essentially sharp.

Our second and main theorem concerns the contribution from $k$-good vectors $\tau \in (0,\infty)^{k+1}$.
 Following \cite{MR4376455}, we introduce the joint spectral measure
\begin{equation}\label{joint spectral measure}
	\mu = \sum_{i_1,i_2,\dots,i_{k+1}} \bigl|\langle e_{i_1} e_{i_2} \cdots e_{i_k}, e_{i_{k+1}} \rangle\bigr|^2 \, \delta_{(\lambda_{i_1},\lambda_{i_2},\dots,\lambda_{i_{k+1}})}
\end{equation}
on $\R^{k+1}$, where $\delta$ denotes the Dirac measure. Thus $\mu$ is a weighted sum of point masses, with weights given by the squared magnitudes of the coefficients $\langle e_{i_1} e_{i_2} \cdots e_{i_k}, e_{i_{k+1}} \rangle$. Our second theorem describes the mass of $\mu$ carried by frequencies $(\lambda_{i_1}, \lambda_{i_2}, \dots, \lambda_{i_{k+1}})$ lying in $k$-good cones.

Before stating the result, we introduce some terminology by recalling the definition and basic properties of the Leray density. Let $F : \R^n \to \R^d$ be a smooth function, and let $y_0 \in \R^d$ be such that the differential $dF$ has full rank at every point of the level set $F^{-1}(y_0)$. By the inverse function theorem, $F^{-1}(y_0)$ is a smooth $(n-d)$-dimensional submanifold. Choose local coordinates $(z,z')$ in a neighborhood of a point of $F^{-1}(y_0)$ so that $z$ parametrizes $F^{-1}(y_0)$ and $\det(\partial F/\partial z') \neq 0$. The Leray density on $F^{-1}(y_0)$ is defined by
\begin{equation}\label{Leray density definition}
	d_F = \left|\det \frac{\partial F}{\partial z'}\right|^{-1} \, dz,
\end{equation}
and is independent of the choice of complementary coordinates $z'$. The corresponding Leray volume is
\[
	\vol F^{-1}(y_0) = \int_{F^{-1}(y_0)} d_F.
\]

Given a continuous function $f$ defined on a neighborhood of $F^{-1}(y_0)$, the integral of $f$ over $F^{-1}(y_0)$ with respect to the Leray measure can be expressed using the Dirac distribution as
\[
	\int_{F^{-1}(y_0)} f \, d_F = \int_{\R^n} f(x) \delta(F(x) - y_0) \, dx.
\]

For the purposes of this paper, we set $F : \R^{kn} \to \R^{k+1}$ by
\begin{equation}\label{def F}
	F(\xi_1,\xi_2,\dots,\xi_k) = \bigl(|\xi_1|,|\xi_2|,\dots,|\xi_k|,\ |\xi_1 + \cdots + \xi_k|\bigr).
\end{equation}
Then the level set $F^{-1}(\tau)$ consists of $k$-tuples $(\xi_1,\xi_2,\dots,\xi_k)$ satisfying
\[
|\xi_j|=\tau_j \ \ (1\le j\le k)
\qquad \text{and} \qquad
|\xi_1+\cdots+\xi_k|=\tau_{k+1}.
\]
Such a $k$-tuple determines a closed $(k+1)$-gon in $\R^n$ with side lengths $\tau_1,\tau_2,\dots,\tau_{k+1}$, where the final side is the closing vector $-(\xi_1+\cdots+\xi_k)$.
The next theorem provides a description of the concentration of $\mu$ in $k$-good cones.

\begin{theorem}\label{main theorem k+1-good}
Let $\rho$ be a Schwartz function on $\R^{k+1}$ with $\int \rho = 1$. Assume that
\[
	\supp \hat \rho \subset (-\inj M, \inj M)^{k+1}.
\]
Let $\Gamma$ be a $k$-good cone. Then for $\tau = (\tau_1, \tau_2, \dots, \tau_{k+1}) \in \Gamma$,
\[
	\rho * \mu(\tau) = (2\pi)^{-kn} \vol M \vol F^{-1}(\tau) + O(|\tau|^{k(n-1)-2}),
\]
where the constants implicit in the big-$O$ notation depend on $M$, $\Gamma$, and $\rho$ but not on $\tau$.
\end{theorem}

Since $\rho$ is rapidly decaying and $\supp \hat \rho$ is contained in $(-\inj M, \inj M)^{k+1}$, the convolution $\rho * \mu$ provides a smoothed version of $\mu$ that captures the mass of $\mu$ near $\tau$. In particular, Theorem \ref{main theorem k+1-good} gives an asymptotic description of this smoothed local mass for $\tau \in \Gamma$. Next, we determine the growth rate of $\vol F^{-1}(\tau)$ as $|\tau| \to \infty$. When $k=2$, this quantity is related to the area of the corresponding triangle \cite{MR4376455}. For $k>2$, the shape of a $(k+1)$-gon with fixed side lengths is not uniquely determined (Figures \ref{2} and \ref{3}), so it is not sufficient to consider only its planar area. We prove in Proposition \ref{Scaling of F} that, under our assumptions, $\vol F^{-1}(\tau)$ is of order $|\tau|^{k(n-1)-1}$, so the remainder term $O(|\tau|^{k(n-1)-2})$ is indeed lower order. As in \cite{MR4376455}, the leading term can be interpreted in terms of the volume of the configuration space of $(k+1)$-gons with side lengths prescribed by $\tau$.
\subsection*{Organization}
In Section~2 we analyze the configuration space that appears in the main term of Theorem~\ref{main theorem k+1-good}, and in particular prove the scaling law for the relevant fiber volume. Section~3 contains the proof of Theorem~\ref{main theorem k+1-gon-bad}, establishing rapid decay of the high-frequency tail in the regime where the $(k+1)$-gon inequality fails, via spectral projectors and a non-stationary phase argument. Section~4 proves Theorem~\ref{main theorem k+1-good} by expressing the Fourier transform of the joint spectral measure in terms of Fourier integral operators and computing the associated symbolic data using clean composition and half-density calculus. Finally, Section~\ref{sec:flat-torus-example} presents a flat-torus example showing that the frequency cutoff in Theorem~\ref{main theorem k+1-gon-bad} is essentially sharp.

\subsection*{Acknowledgement} This project is supported by the National Key R\&D Program of China under Grant No. 2022YFA1007200, the Natural Science Foundation of China under Grant No. 12571107, and the Zhejiang Provincial Natural Science Foundation of China under Grant No. LR25A010001.

\section{Order of the main term in Theorem \ref{main theorem k+1-good}}

The leading term in Theorem~\ref{main theorem k+1-good} involves the quantity $\vol F^{-1}(\tau)$, which can be viewed as the Leray volume of the configuration space of closed $(k+1)$-gons with side lengths prescribed by $\tau$. To see that this term is well-defined and to compare it with the remainder, we need basic geometric information about the fibers of $F$ when $\tau$ lies in a $k$-good cone. In particular, we must know that $F^{-1}(\tau)$ is a smooth compact manifold (so the Leray volume is finite and nonzero) and that this volume has a precise scaling law under dilations of $\tau$. The next proposition establishes these properties; the key input is that the nondegeneracy condition \eqref{2.2} rules out singular fibers.

\begin{proposition}
 [Scaling of $F$]\label{Scaling of F} Let $\Gamma$ be a $k$-good cone and fix $\tau \in \Gamma$. Then $F^{-1}(\tau) \neq \emptyset$ and $\tau$ is a regular value of $F$, so $F^{-1}(\tau)$ is a smooth compact submanifold of codimension $k+1$. Let
\[
d:=k n-(k+1)=k(n-1)-1
\]
be its dimension. Then $0<\operatorname{vol} F^{-1}(\tau)<\infty$ and, for every $r>0$,
$$
\operatorname{vol} F^{-1}(r \tau)=r^d \operatorname{vol} F^{-1}(\tau).
$$
\end{proposition}

\begin{proof}
Fix $\tau=(\tau_1,\ldots,\tau_{k+1})\in\Gamma$. Recall that
\[
F(\xi_1,\ldots,\xi_k)=\bigl(|\xi_1|,\ldots,|\xi_k|,\ |\xi_1+\cdots+\xi_k|\bigr).
\]

We begin by showing that the fiber $F^{-1}(\tau)$ is nonempty. Since $\tau\in\Gamma\subset\Gamma_g$, the strict $(k+1)$-gon inequalities \eqref{2.1} hold. Hence there exist vectors
$u_1,\ldots,u_{k+1}\in\R^n$ (for instance, all lying in a fixed $2$-plane) such that
\[
|u_j|=\tau_j\ (1\le j\le k+1),\qquad u_1+\cdots+u_{k+1}=0.
\]
Setting $\xi_j:=u_j$ for $1\le j\le k$ gives $|\xi_j|=\tau_j$ and
\[
|\xi_1+\cdots+\xi_k|=|-(u_{k+1})|=\tau_{k+1},
\]
so $F(\xi_1,\ldots,\xi_k)=\tau$ and $F^{-1}(\tau)\neq\emptyset$.

Next we verify that $\tau$ is a regular value of $F$. Fix $\xi=(\xi_1,\ldots,\xi_k)\in F^{-1}(\tau)$ and set
\[
\omega_j=\frac{\xi_j}{|\xi_j|}\in S^{n-1}\ (1\le j\le k),\qquad
\omega_{k+1}=\frac{\xi_1+\cdots+\xi_k}{|\xi_1+\cdots+\xi_k|}\in S^{n-1}.
\]
For $h=(h_1,\ldots,h_k)\in\R^{kn}$ we compute
\[
d(|\xi_j|)[h]=\langle \omega_j,h_j\rangle,\qquad
d(|\xi_1+\cdots+\xi_k|)[h]=\Bigl\langle \omega_{k+1},\sum_{j=1}^k h_j\Bigr\rangle.
\]
If $\mathrm{rank}(dF_\xi)<k+1$, these $k+1$ linear functionals are linearly dependent, so there exist
$(c_1,\ldots,c_k,c_{k+1})\neq 0$ such that for all $h$,
\[
\sum_{j=1}^k c_j\langle \omega_j,h_j\rangle
+c_{k+1}\Bigl\langle \omega_{k+1},\sum_{j=1}^k h_j\Bigr\rangle=0.
\]
Equivalently,
\[
\sum_{j=1}^k \langle c_j\omega_j+c_{k+1}\omega_{k+1},\ h_j\rangle=0
\quad\text{for all }h_1,\ldots,h_k,
\]
hence $c_j\omega_j+c_{k+1}\omega_{k+1}=0$ for every $j$. In particular $c_{k+1}\neq 0$, and
$\omega_j=\varepsilon_j\,\omega_{k+1}$ for some $\varepsilon_j\in\{\pm1\}$. Therefore
\[
\xi_1+\cdots+\xi_k=\sum_{j=1}^k |\xi_j|\omega_j
=\Bigl(\sum_{j=1}^k \varepsilon_j|\xi_j|\Bigr)\omega_{k+1},
\]
and taking norms gives
\[
\tau_{k+1}=|\xi_1+\cdots+\xi_k|
=\Bigl|\sum_{j=1}^k \varepsilon_j|\xi_j|\Bigr|
=\Bigl|\sum_{j=1}^k \varepsilon_j\,\tau_j\Bigr|.
\]
Let $s:=\sum_{j=1}^k \varepsilon_j\,\tau_j$ and choose $\varepsilon_{k+1}\in\{\pm1\}$ so that
$\varepsilon_{k+1}\tau_{k+1}+s=0$. Then the sign vector
\[
\varepsilon=(\varepsilon_1,\ldots,\varepsilon_k,\varepsilon_{k+1})\in\{\pm1\}^{k+1}
\]
satisfies $\varepsilon\cdot\tau=0$, contradicting the nondegeneracy condition \eqref{2.2}.
Hence $\mathrm{rank}(dF_\xi)=k+1$ for all $\xi\in F^{-1}(\tau)$, so $\tau$ is a regular value.

It follows that $F^{-1}(\tau)$ is a smooth submanifold of codimension $k+1$ in $\R^{kn}$, hence has
dimension $d=kn-(k+1)=k(n-1)-1$. Moreover, it is compact since
\[
F^{-1}(\tau)\subset \prod_{j=1}^k\{|\xi_j|=\tau_j\},
\]
a product of spheres. In particular, the Leray density $d_F$ is smooth and strictly positive on $F^{-1}(\tau)$,
so $0<\vol F^{-1}(\tau)<\infty$.

Finally we prove the scaling law. Noting that
\[
\vol F^{-1}(\tau)=\int_{\R^{kn}}\delta(F(\xi)-\tau)\,d\xi.
\]
For $r>0$, change variables $\xi=r\eta$. Since $F(r\eta)=rF(\eta)$ and $\delta(r y)=r^{-(k+1)}\delta(y)$ in
$\R^{k+1}$, we obtain
\begin{align*}
\vol F^{-1}(r\tau)
&=\int_{\R^{kn}}\delta(F(\xi)-r\tau)\,d\xi \\
&=\int_{\R^{kn}}\delta(rF(\eta)-r\tau)\,r^{kn}\,d\eta \\
&=r^{kn-(k+1)}\int_{\R^{kn}}\delta(F(\eta)-\tau)\,d\eta \\
&=r^{d}\,\vol F^{-1}(\tau),
\end{align*}
where $d=kn-(k+1)=k(n-1)-1$.
\end{proof}

\section{Proof of Theorem \ref{main theorem k+1-gon-bad}}

To analyze the rapidly decaying part of the sum \eqref{single sums}, we introduce
a smooth spectral projection operator. Fix $\chi\in \mathcal S(\R)$ with
$\chi(0)=1$, and define
\[
\chi_\lambda := \chi\!\left(\lambda-\sqrt{-\Delta_g}\right).
\]
If $e_\lambda$ is an eigenfunction with frequency $\lambda$, then
$\chi_\lambda e_\lambda = e_\lambda$. Moreover, if we assume that
$\supp \widehat{\chi}\subset (1,2)$, then (modulo $O(\lambda^{-N})$ errors for
every $N$, as in \cite[Lemma 5.1.3]{MR3645429}) we have the oscillatory integral
representation
\begin{equation}\label{2.4}
\chi_\lambda e_\lambda(x)
= \lambda^{\frac{n-1}{2}}
\int_M e^{i\lambda \psi(x,y)}\, a_\lambda(x,y)\, e_\lambda(y)\, dV(y),
\end{equation}
where $\psi(x,y)=d_g(x,y)$ is the Riemannian distance function. Here we assume,
as we may, that the injectivity radius is $10$ or more, and the amplitude
$a_\lambda$ lies in a bounded subset of $C^\infty$ and satisfies
\[
a_\lambda(x,y)=0 \qquad \text{if } d_g(x,y)\notin (1,2).
\]

Before proving Theorem \ref{main theorem k+1-gon-bad}, we record a uniform
nonstationary-phase estimate for the coefficients in the high-output-frequency regime.

\begin{lemma}\label{lem:kplus1-eigenfunction-product}
Let $\epsilon>0$ and set
\[
T_m := (1+\epsilon)(\tau_1+\cdots+\tau_k)+m ,\qquad m\in\N.
\]
Assume that $\lambda_{i_j}\le \tau_j$ for $1\le j\le k$ and that
\[
\lambda_{i_{k+1}}\in [T_m,T_m+1].
\]
Then for every integer $N\ge 0$ there exists a constant $C_{\epsilon,N}$ such that
\[
\bigl|\langle e_{i_1} e_{i_2}\cdots e_{i_k},\, e_{i_{k+1}}\rangle\bigr|
\le C_{\epsilon,N}\, T_m^{-N}.
\]
\end{lemma}

\begin{proof}[Proof of Theorem \ref{main theorem k+1-gon-bad}]
Decompose the tail into unit spectral windows:
\begin{equation*}
\sum_{\lambda_{i_{k+1}}\ge (1+\epsilon)(\tau_1+\cdots+\tau_k)}
\bigl|\langle e_{i_1}\cdots e_{i_k}, e_{i_{k+1}}\rangle\bigr|^2
=\sum_{m=0}^\infty \sum_{\lambda_{i_{k+1}}\in[T_m,T_m+1]}
\bigl|\langle e_{i_1}\cdots e_{i_k}, e_{i_{k+1}}\rangle\bigr|^2 
\end{equation*}
By the Weyl law,
\[
\#\{\lambda_{i_{k+1}}\in[T_m,T_m+1]\}\ \lesssim\ T_m^{n-1}.
\]
Lemma \ref{lem:kplus1-eigenfunction-product} gives
\[
\sup_{\lambda_{i_{k+1}}\in[T_m,T_m+1]}
\bigl|\langle e_{i_1}\cdots e_{i_k}, e_{i_{k+1}}\rangle\bigr|^2
\ \lesssim_{\epsilon,N}\ T_m^{-2N}.
\]
Hence
\[
\sum_{\lambda_{i_{k+1}}\ge (1+\epsilon)(\tau_1+\cdots+\tau_k)}
\bigl|\langle e_{i_1}\cdots e_{i_k}, e_{i_{k+1}}\rangle\bigr|^2
\ \lesssim_{\epsilon,N}\ \sum_{m=0}^\infty T_m^{\,n-1-2N}.
\]
Given $N'>0$, choose $N$ so that $2N-(n-1) > N'$. Then the last sum is bounded by
$C_{\epsilon,N'}(\tau_1+\cdots+\tau_k)^{-N'}$, which proves the theorem.
\end{proof}

\begin{proof}[Proof of Lemma \ref{lem:kplus1-eigenfunction-product}]
Write
\[
\langle e_{i_1}\cdots e_{i_k}, e_{i_{k+1}}\rangle
=\int_M e_{i_1}(x)\cdots e_{i_k}(x)\,\overline{e_{i_{k+1}}(x)}\,dV(x).
\]
Using the oscillatory representation \eqref{2.4} for each factor (and its complex
conjugate for $\overline{e_{i_{k+1}}}$), we can express this integral as a
$(k+2)$-fold integral with phase
\[
\phi(x,y_1,\ldots,y_{k+1})
=\sum_{j=1}^k \frac{\lambda_{i_j}}{\lambda_{i_{k+1}}}\,\psi(x,y_j)\;-\;\psi(x,y_{k+1}),
\]
where $\psi(x,y)=d_g(x,y)$ and all variables are restricted to the region
$1<d_g(x,y_j)<2$ by the support of the amplitudes $a_{\lambda_{i_j}}$.

On this region, $|\nabla_x\psi(x,y)|=1$, so if $\nabla_x\phi=0$ then
\[
\nabla_x\psi(x,y_{k+1})
=\sum_{j=1}^k \frac{\lambda_{i_j}}{\lambda_{i_{k+1}}}\,\nabla_x\psi(x,y_j),
\]
which is impossible because the left-hand side has norm $1$ while the right-hand
side has norm at most $\sum_{j=1}^k \lambda_{i_j}/\lambda_{i_{k+1}}$.
Under the assumptions $\lambda_{i_j}\le\tau_j$ and $\lambda_{i_{k+1}}\ge T_m$ we have
\[
\sum_{j=1}^k \frac{\lambda_{i_j}}{\lambda_{i_{k+1}}}
\le \frac{\tau_1+\cdots+\tau_k}{(1+\epsilon)(\tau_1+\cdots+\tau_k)+m}
\le \frac{1}{1+\epsilon},
\]
so $|\nabla_x\phi|$ is bounded below by a positive constant depending only on
$\epsilon$.

With this uniform nonstationarity, we integrate by parts in $x$.
Each integration by parts gains a factor $\lambda_{i_{k+1}}^{-1}$, while the
amplitudes $a_{\lambda_{i_j}}$ have uniform $C^\infty$ bounds on their supports.
Therefore, for every $N$ we obtain
\[
\bigl|\langle e_{i_1}\cdots e_{i_k}, e_{i_{k+1}}\rangle\bigr|
\le C_{\epsilon,N}\,\lambda_{i_{k+1}}^{-N}.
\]
Finally, on the window $\lambda_{i_{k+1}}\in[T_m,T_m+1]$ we have
$\lambda_{i_{k+1}}\simeq T_m$, so the bound becomes
$\bigl|\langle e_{i_1}\cdots e_{i_k}, e_{i_{k+1}}\rangle\bigr|
\le C_{\epsilon,N} T_m^{-N}$, as claimed.
\end{proof}

\section{The proof of Theorem \ref{main theorem k+1-good}}

Next, we express $\check\mu\,|dt|^{1/2}$ as a composition of Fourier integral operators and Lagrangian distributions. We then compute their orders and principal symbols to obtain the symbolic data of $\check\mu\,|dt|^{1/2}$. After multiplying by the cutoff $\check\rho\,|dt|^{1/2}$, we apply the inverse Fourier transform to prove Theorem \ref{main theorem k+1-good}.

To compute the symbolic data of these compositions, we use the clean composition calculus of Duistermaat and Guillemin \cite{MR405514}. We also use a half-density formula for composed Fourier integral operators due to the first author \cite{MR4376455}. For background on Fourier integral operators and half-densities, see \cite{MR451313,MR388464,MR4376455}.

For convenience, we fix once and for all an $L^2$-orthonormal eigenbasis
$\{e_j\}_{j\ge1}$ consisting of real-valued eigenfunctions. With this convention, we may write
\[
	\mu = \sum_{i_1,i_2,\dots,i_{k+1}}
	\left|\int_M e_{i_1}(x)\, e_{i_2}(x)\cdots e_{i_{k+1}}(x)\, dV_M(x)\right|^2
	\delta_{(\lambda_{i_1},\lambda_{i_2},\dots,\lambda_{i_{k+1}})} .
\]

Let $P$ denote the elliptic, first-order pseudodifferential operator $\sqrt{-\Delta_g}$ on $M$. To simplify notation, set
$m=(i_1,i_2,\dots,i_{k+1})\in\N^{k+1}$ and define the smooth half-density
\[
	\phi_m
	= e_{i_1}\otimes e_{i_2}\otimes \cdots \otimes e_{i_{k+1}}\,|dV_{M^{k+1}}|^{1/2}
\]
on $M^{k+1}$. The collection $\{\phi_m\}_{m\in\N^{k+1}}$ forms a Hilbert basis for the intrinsic $L^2$ space on $M^{k+1}$, and each $\phi_m$ is a joint eigen-half-density for the commuting operators obtained by letting $P$ act in one factor and the identity act in the others. We write its joint eigenvalues as
$\lambda_m=(\lambda_{i_1},\lambda_{i_2},\dots,\lambda_{i_{k+1}})$.

Let $\Delta=\{(x,x,\dots,x):x\in M\}\subset M^{k+1}$ denote the diagonal, and let $\delta_\Delta$ be the corresponding half-density distribution defined by
\[
	(\delta_\Delta, f\,|dV_{M^{k+1}}|^{1/2})
	= \int_M f(x,x,\dots,x)\, dV_M(x)
	\qquad \text{for all $f\in C^\infty(M^{k+1})$.}
\]
With this notation,
\[
	\int_M e_{i_1} e_{i_2}\cdots e_{i_{k+1}}\, dV_M
	= (\delta_\Delta,\phi_m).
\]

Let $\check\mu$ denote the inverse Fourier transform of $\mu$ on $\R^{k+1}$. With our normalization,
\[
	\check \mu(t) = \frac{1}{(2\pi)^{k+1}} \sum_{m} |(\delta_\Delta, \phi_m)|^2\, e^{i\langle t, \lambda_m \rangle},
	\qquad t=(t_1,\dots,t_{k+1})\in\R^{k+1}.
\]
Equivalently, if we view $\check\mu(t)\,|dt|^{1/2}$ as a half-density distribution on $\R^{k+1}$, then
\begin{align*}
	\check \mu(t)\,|dt|^{1/2}
	&= \frac{1}{(2\pi)^{k+1}} \sum_{m} (\delta_\Delta \otimes \delta_\Delta,\ \phi_m \otimes \phi_m)\,
	e^{i\langle t, \lambda_m \rangle}\,|dt|^{1/2} \\
	&= \frac{1}{(2\pi)^{k+1}}\, U \circ \delta_{\Delta\times \Delta}(t),
\end{align*}
where $\Delta\times\Delta\subset M^{k+1}\times M^{k+1}$ and $\delta_{\Delta\times\Delta}:=\delta_\Delta\otimes\delta_\Delta$.
Here $U$ is the operator from test half-densities on $M^{k+1}\times M^{k+1}$ to half-density distributions on $\R^{k+1}$ with kernel
\[
	U(t,x,y) = e^{it_1P}(x_1,y_1)\, e^{it_2P}(x_2,y_2)\cdots e^{it_{k+1}P}(x_{k+1},y_{k+1}),
\]
for $x=(x_1,\dots,x_{k+1})\in M^{k+1}$ and $y=(y_1,\dots,y_{k+1})\in M^{k+1}$. Here each $e^{it_jP}$ denotes the half-wave operator.

Therefore, our main task is to study certain properties of the composition
$U \circ \delta_{\Delta \times \Delta}$. However, without any localization in the
joint spectrum, this composition need not be clean in the sense of Duistermaat
and Guillemin \cite{MR405514}. To localize the eigenvalue variables to a
$k$-good region, we insert a truncation operator $A$ by applying it to one copy
of $\delta_\Delta$ in $\delta_\Delta \otimes \delta_\Delta$. Here
$A \in \Psi^0_{\text{phg}}(M^{k+1})$ is defined spectrally by
\begin{equation}\label{def A pdo}
	A\phi_m = a(\lambda_m)\phi_m,
\end{equation}
where $a$ is real-valued, smooth, and positive-homogeneous of order $0$ outside,
say, the ball $B$ of radius $1$. By the conical support of $a$, we mean the
smallest closed cone $\Gamma \subset \R^{k+1}\setminus 0$ for which
$\supp a \setminus B \subset \Gamma$. This localization plays an important role
in the composition $U \circ \delta_{\Delta \times \Delta}$.

Next, we examine its effect on the calculation of $\check \mu(t)\,|dt|^{1/2}$.
Define the modified measure
\[
	\mu_a = \sum_m a(\lambda_m)\,|(\delta_\Delta, \phi_m)|^2\,\delta_{\lambda_m}.
\]
Tracing back through the computations above, we obtain
\begin{equation}\label{strategy main composition}
	\check \mu_a(t)\,|dt|^{1/2}
	= \frac{1}{(2\pi)^{k+1}}\, U \circ (\delta_\Delta \otimes A\delta_\Delta)(t).
\end{equation}
In particular, if $a \equiv 1$ on a $k$-good cone, then $\mu_a$ and $\mu$ agree
on that cone. Since in Theorem \ref{main theorem k+1-good} we fix a $k$-good cone
$\Gamma$, it suffices to study the operator $U \circ (\delta_\Delta \otimes
A\delta_\Delta)$.

The core of our argument is the computation of the symbolic data of the
composition \eqref{strategy main composition}. Studying this operator directly
is cumbersome. Instead, we rewrite it as a composition of four simpler pieces.
Let $Q : C^\infty(M^{k+1}) \to C^\infty(\R^{k+1} \times M^{k+1})$ be the operator
with the same distribution kernel as $U$, and let
$S : C^\infty(M^{k+1}) \to C^\infty(\R^{k+1})$ be the operator with distribution
kernel
\begin{equation}\label{strategy first composition}
	(Q \circ A) \circ \delta_\Delta.
\end{equation}
It follows that
\begin{equation}\label{strategy second composition}
	U \circ (\delta_\Delta \otimes A\delta_\Delta) = S \circ \delta_\Delta.
\end{equation}
We first verify that $Q\circ A$ and $\delta_\Delta$ are Fourier integral
operators, and compute their canonical relations and principal symbols (including
the half-density parts). Following \cite{MR4376455}, we then compose $Q\circ A$
with $\delta_\Delta$ to obtain the symbolic data of $S$, and finally compute the
symbolic data of $S\circ \delta_\Delta$. This decomposition is used only to
simplify the symbol computations for each component.

\begin{proposition}\label{delta symbol}
$\delta_\Delta \in I^{(k-1)n/4}(M^{k+1}, \dot N^*\Delta)$ where, in canonical
local coordinates,
\[
	\dot N^*\Delta = \{(x_1,-\xi_2-\xi_3-\dots -\xi_{k+1}, x_1, \xi_2, x_1, \xi_3,
	\dots ,x_1 , \xi_{k+1}) : x_1 \in M,\ \xi \in \dot T^*_{(x_1,\dots ,x_1)}M^{k+1} \},
\]
with principal symbol having half-density part
\[
	(2\pi)^{-(k-1)n/4} \frac{1}{|g(x_1)|^{(k-1)/4}}
	|dx_1 \, d\xi_2 \, d\xi_3 \dots \, d\xi_{k+1}|^{1/2}.
\]
\end{proposition}

Let $X$ be an open subset of $\R^n$, and let $(x,\theta)\in X\times(\R^N\setminus 0)$.
Recall that a homogeneous phase function $\varphi(x,\theta)$ is called
nondegenerate if $d\varphi \neq 0$ and, whenever $\varphi'_{\theta}(x,\theta)=0$,
the differentials $d(\varphi'_{\theta_1}),\dots,d(\varphi'_{\theta_N})$ are
linearly independent. This implies that the critical set
\[
	C_{\varphi} = \{(x,\theta): \varphi'_{\theta}(x,\theta)=0\}
\]
is a smooth submanifold of $X \times (\R^{N}\setminus 0)$. Consequently,
\[
	\Lambda_{\varphi}
	= \{(x, \varphi'_x(x,\theta)) : (x,\theta) \in C_{\varphi}\}
\]
is a Lagrangian submanifold. Before proving Proposition \ref{delta symbol}, we
record the following lemma about Fourier integral operators and half-densities,
as in \cite[Lemma 6.1.4]{MR3645429}.

\begin{lemma}
Let $\varphi$ be a nondegenerate phase function in an open conic neighborhood of
$(x_0,\theta_0)\in \R^n\times (\R^N\setminus 0)$. Define a half-density
distribution $u$ on $X$ by the oscillatory integral
\[
	u(x) = (2\pi)^{-(n + 2N)/4}
	\left( \int_{\R^N} e^{i\varphi(x,\theta)} a(x,\theta) \, d\theta \right)
	|dx|^{1/2},
\]
where $a(x,\theta)$ belongs to the symbol class
$S^m_{\text{phg}}(\R^n \times (\R^N \setminus 0))$. Then $u$ is a Lagrangian
distribution associated to $\Lambda_\varphi$ of order
\[
	\ord u = m + (2N - n)/4,
\]
and we write $u \in I^{m + (2N - n)/4}(X, \Lambda_\varphi)$.
\end{lemma}

To $u$ one associates a homogeneous half-density on $\Lambda_\varphi$. Together
with the Maslov index, this comprises the principal symbol of $u$. We only
describe the half-density part. Given a parametrization of $C_\varphi$ by
$\lambda$ in a subset of $\R^n$, the half-density part of the principal symbol is
given by transporting
\[
	a_0(\lambda)\, \sqrt{d_{\varphi_\theta'}}
	= a_0(\lambda)\left|\det \frac{\partial(\lambda, \varphi'_\theta)}{\partial(x,\theta)}\right|^{-1/2}
	|d\lambda|^{1/2}
\]
to $\Lambda_\varphi$ via the map $C_\varphi \to \Lambda_\varphi$, where $a_0$
denotes the top-order term of $a$ and $d_{\varphi'_\theta}$ is the Leray density
on $C_\varphi = (\varphi'_\theta)^{-1}(0)$.

\begin{proof}[Proof of Proposition \ref{delta symbol}]
Let $x$ denote local coordinates on $M$, and let $(x_1,x_2,\dots,x_{k+1})$ denote
the induced coordinates on $M^{k+1}$. Then $(x_1,x_2,\dots,x_{k+1})$ parametrizes
a neighborhood of $M^{k+1}$ intersecting the diagonal. Suppose that $f$ is smooth
and supported in such a coordinate patch. By the Fourier inversion formula,
\begin{multline*}
	(\delta_\Delta, f|dV_{M^{k+1}}|^{1/2})
	= \int_{\R^n} f(x_1,\dots,x_1)\,|g(x_1)|^{1/2}\,dx_1 \\
	= (2\pi)^{-kn} \idotsint e^{i\phi(x_1,\dots,x_{k+1},\xi_2,\dots,\xi_{k+1})}
	f(x_1,\dots,x_{k+1})\,|g(x_1)|^{1/2}\,
	dx_1 \dots dx_{k+1}\, d\xi_2 \dots d\xi_{k+1},
\end{multline*}
where
\[
\phi(x_1,\dots,x_{k+1},\xi_2,\dots,\xi_{k+1})
= \langle x_1-x_2,\xi_2\rangle + \langle x_1-x_3,\xi_3\rangle
+ \dots + \langle x_1-x_{k+1},\xi_{k+1}\rangle.
\]
Recall that in local coordinates,
\[
f|dV_{M^{k+1}}|^{1/2}
= f(x_1,\dots,x_{k+1})\,|g(x_1)|^{1/4}\dots |g(x_{k+1})|^{1/4}
\,|dx_1 \dots dx_{k+1}|^{1/2}.
\]
Hence we may write
\begin{multline*}
	\delta_\Delta
	= (2\pi)^{-kn}\left(\int_{\R^n}\dots\int_{\R^n}
	e^{i\phi}\,
	\frac{|g(x_1)|^{1/4}}{|g(x_2)|^{1/4}\dots |g(x_{k+1})|^{1/4}}
	\,d\xi_2 \dots d\xi_{k+1}\right)|dx_1 \dots dx_{k+1}|^{1/2} \\
	= (2\pi)^{-(3kn+n)/4}\left(\int_{\R^{kn}} e^{i\phi}\,
	(2\pi)^{-(kn-n)/4}\,G(x_1,\dots,x_{k+1})\, d\xi_2 \dots d\xi_{k+1}\right)
	|dx_1 \dots dx_{k+1}|^{1/2},
\end{multline*}
where
\[
G(x_1,\dots,x_{k+1})
= \frac{|g(x_1)|^{1/4}}{|g(x_2)|^{1/4}\dots |g(x_{k+1})|^{1/4}}.
\]
Here we have pulled a factor of $(2\pi)^{-(kn-n)/4}$ into the symbol so that the
normalization matches the convention in the preceding lemma.

We now check that $\phi(x_1,\dots,x_{k+1},\xi_2,\dots,\xi_{k+1})$ is a
nondegenerate phase function. First, $\phi$ is positive-homogeneous in the
frequency variables $(\xi_2,\dots,\xi_{k+1})$. Next, writing
$x=(x_1,\dots,x_{k+1})$ and $\theta=(\xi_2,\dots,\xi_{k+1})$, we have
\[
	\phi_x' =
		\begin{bmatrix}
		\xi_2 + \cdots + \xi_{k+1} \\
		-\xi_2 \\
		-\xi_3 \\
        \vdots \\
        -\xi_{k+1}
		\end{bmatrix}
	\qquad \text{and} \qquad
	\phi_\theta' =
		\begin{bmatrix}
		x_1 - x_2 \\
        x_1 - x_3 \\
        \vdots \\
	 	x_1 - x_{k+1}
		\end{bmatrix}.
\]

so $d\phi \neq 0$ whenever $(\xi_2,\dots,\xi_{k+1}) \neq (0,\dots,0)$. The
critical set is
\[
	C_\phi
	= \{(x_1,\dots,x_1,\xi_2,\dots,\xi_{k+1}) : x_1 \in \R^n,\ 
	(\xi_2,\dots,\xi_{k+1}) \in \R^{kn}\setminus 0\},
\]
and on $C_\phi$ we have
\[
	d\phi_\theta' =
	\begin{bmatrix}
		I & -I & 0 & 0 & \cdots & 0 \\
		I & 0 & -I & 0 & \cdots & 0 \\
        I & 0 & 0 & -I & \cdots & 0 \\
        \vdots & \vdots & \vdots & \vdots & \ddots & \vdots \\
        I & 0 & 0 & 0 & \cdots & -I
	\end{bmatrix}.
\]

which has full rank. Hence $\delta_\Delta$ is a Lagrangian distribution of order
\[
	\ord(\delta_\Delta) = m+\frac{2N-n}4 = \frac{2kn-(k+1)n}4 = \frac{(k-1)n}4,
\]
associated with the image of $C_\phi$ under the map
\[
(x_1,\dots,x_{k+1},\xi_2,\dots,\xi_{k+1})
\mapsto (x_1,\dots,x_{k+1},\phi_x'(x_1,\dots,x_{k+1},\xi_2,\dots,\xi_{k+1})),
\]
namely
\[
	\{(x_1, x_1,\dots,x_1; \xi_2 + \dots +\xi_{k+1}, -\xi_2, \dots ,-\xi_{k+1})
	\in \dot T^* M^{k+1} : x_1 \in M,\ \xi_2,\dots,\xi_{k+1} \in T_{x_1}^*M\},
\]
which is the punctured conormal bundle $\dot N^*\Delta$.

The half-density part of the symbol is obtained by transporting
\[
	(2\pi)^{-n/4} G(x_1,\dots,x_1)\, d_{\phi'_\theta}^{1/2}
\]
via the parametrization of $\dot N^*\Delta$ by $(x_1,\xi_2,\dots,\xi_{k+1})$,
where $d_{\phi'_\theta}$ is the Leray density on $C_\phi = (\phi'_\theta)^{-1}(0)$.
In these coordinates,
\begin{align*}
    d_{\phi'_\theta}
    &= \left|\det \frac{\partial(x_1,\xi_2,\dots,\xi_{k+1}, \phi'_\theta)}{\partial(x_1,x_2,\dots,x_{k+1},\theta)}\right|^{-1}
    \, dx_1 \, d\xi_2 \dots d\xi_{k+1} \\
    &= dx_1 \, d\xi_2 \dots d\xi_{k+1}.
\end{align*}
Therefore, modulo a Maslov factor, the principal symbol has half-density part
\[
	(2\pi)^{-(k-1)n/4}\frac{1}{|g(x_1)|^{(k-1)/4}}
	|dx_1 \, d\xi_2 \, d\xi_3 \dots \, d\xi_{k+1}|^{1/2}.
\]
The proposition follows after negating $\xi_2,\xi_3,\dots,\xi_{k+1}$.
\end{proof}

Next, we describe the symbolic data of $Q \circ A$. For this, we recall the
Hamilton flow. Let $X$ be a smooth manifold and let $p$ be a smooth function on
$\dot T^*X$. The Hamilton vector field $H_p$ is defined locally by
\begin{equation}\label{local hamiltonian}
	H_p = \sum_j \left( \frac{\partial p}{\partial \xi_j} \frac{\partial}{\partial x_j}
	- \frac{\partial p}{\partial x_j} \frac{\partial}{\partial \xi_j} \right),
\end{equation}
where $(x_1,\ldots,x_n,\xi_1,\ldots,\xi_n)$ are canonical local coordinates on
$T^*X$. Equivalently, $H_p$ is characterized by the identity
$\omega(H_p,v)=dp(v)$ for all vector fields $v$ on $\dot T^*X$, where
$\omega=d\xi\wedge dx$ is the symplectic $2$-form on $T^*X$. The flow
$\exp(tH_p)$ is called the Hamilton flow of $p$.

Geometrically, if $(x,\xi)\in T^*X$ and we write
$\exp(tH_p)(x,\xi)=(x(t),\xi(t))$, then $x(t)$ is the projected trajectory on
$X$, with cotangent component $\xi(t)$ transported along it. If $p$ is
positive-homogeneous of order $1$, then the Hamilton flow is homogeneous in the
sense that it commutes with fiber dilations: for $\lambda>0$,
\[
	\exp(tH_p)(x,\lambda \xi) = (x(t), \lambda \xi(t)).
\]
For further background on Hamilton vector fields, flows, and the symplectic
geometry underlying Fourier integral operators, see \cite{MR451313} and
\cite{MR3645429}.

Returning to our problem, write
$(x,\xi)=(x_1,\xi_1,\dots,x_{k+1},\xi_{k+1})\in T^*M^{k+1}$ with each $\xi_j\neq 0$,
and define, for $t=(t_1,\dots,t_{k+1})\in\R^{k+1}$,
\[
	G^t(x,\xi)
	= \bigl(\exp(t_1H_p)(x_1,\xi_1), \dots, \exp(t_{k+1}H_p)(x_{k+1},\xi_{k+1})\bigr),
\]
where $H_p$ is the Hamilton vector field of the principal symbol $p$ of
$P=\sqrt{-\Delta_g}$ on $M$. We also write
\[
	p(x,\xi) = (p(x_1,\xi_1), \dots, p(x_{k+1},\xi_{k+1})) \in \R^{k+1}.
\]
With this notation, the following proposition gives the symbol of $Q \circ A$.

\begin{proposition}\label{Q of A is an FIO}
If the conical support of $a$ is $k$-good, then $Q \circ A$ in
\eqref{strategy first composition} is a Fourier integral operator in
$I^{-(k+1)/4}(\R^{k+1} \times M^{k+1} \times M^{k+1}, \mathcal C')$ associated to
the canonical relation
\[
	\mathcal C
	= \{ (t, p(x,\xi), G^{-t}(x,\xi); x, \xi ) : t \in \R^{k+1},\ (x,\xi) \in \esssupp A \},
\]
having principal symbol with half-density part
\[
	(2\pi)^{(k+1)/4} a(p(x,\xi)) |dt \, dx \, d\xi|^{1/2}.
\]
\end{proposition}

\begin{proof}
We recall some standard facts about the half-wave operator and its wavefront set,
for example from \cite[\S 29.1]{MR1481433}. For $t \in \R$, the operator
$e^{-itP} : C^\infty(M) \to C^\infty(\R \times M)$ is a Fourier integral operator
of order $-1/4$ with canonical relation
\[
	\{(t,-p(x,\xi), \exp(tH_p)(x,\xi), x, \xi) : t \in \R,\ (x,\xi) \in \dot T^*M \}
\]
and principal symbol with half-density part
\[
	(2\pi)^{1/4} |dt \, dx \, d\xi|^{1/2}.
\]
After the change of variables $t \mapsto -t$, the operator $e^{itP}$ has the same
order and symbol, and its canonical relation becomes
\[
	\{(t, p(x,\xi), \exp(-tH_p)(x,\xi); x, \xi) : t \in \R,\ (x,\xi) \in \dot T^*M \}.
\]

Let $u$ denote the distribution kernel of $e^{itP}$ on $\R \times M \times M$.
Returning to the multi-parameter notation $t \in \R^{k+1}$ and
$(x,\xi) \in T^*M^{k+1}$, the kernel of $Q$ is, up to a permutation of
variables, the $(k+1)$-fold tensor product $u \otimes \cdots \otimes u$.
Therefore, by \cite[Proposition 1.3.5]{MR451313}, the wavefront set of the kernel
of $Q$ is contained in the union of
\begin{equation}\label{U lagrangian submanifold}
	\WF(u) \times \dots \times \WF(u)
	\simeq \{(t, p(x,\xi), G^{-t}(x,\xi), x, -\xi) : t \in \R^{k+1},\ (x,\xi) \in \dot T^*M^{k+1} \}
\end{equation}
together with the additional components coming from products in which at least
one factor is the zero section of $T^*(\R \times M \times M)$. Concretely, these
are the products $X_1 \times \cdots \times X_{k+1}$ where each $X_j$ is either
$\WF(u)$ or the zero section, and at least one $X_j$ is the zero section.

It remains to check that composing with $A$ removes these extra components. Since
$(p(x_1,\xi_1),\dots,p(x_{k+1},\xi_{k+1}))$ lies in a $k$-good cone for any
$(x,\xi)\in \esssupp A$, each coordinate $p(x_j,\xi_j)$ is positive, hence each
$\xi_j$ is nonzero. In particular, none of the zero-section factors can occur
after applying $A$. Thus
\[
	\WF(Q \circ A) \subset \{(t, p(x,\xi), G^{-t}(x,\xi); x, -\xi) : t \in \R^{k+1},\ (x,\xi) \in \esssupp A \}.
\]
On this canonical relation, the principal symbol is the product of the symbols
of the half-wave factors and the principal symbol of $A$, namely
\[
	(2\pi)^{(k+1)/4} a(p(x,\xi)) |dt \, dx \, d\xi|^{1/2},
\]
modulo a Maslov factor.
\end{proof}

Combining Proposition \ref{delta symbol} and Proposition \ref{Q of A is an FIO},
we now state the symbolic data of the composition $(Q \circ A) \circ \delta_\Delta$
in the case where the conical support of $a$ is $k$-good.

\begin{proposition}\label{prop first composition}
Suppose the conical support of $a$ is $k$-good. Then $(Q \circ A) \circ \delta_\Delta$
belongs to the class $I^{(kn-n-k-1)/4}(\R^{k+1} \times M^{k+1}; \Lambda)$, where
\[
	\Lambda
	= \mathcal C \circ \dot N^*\Delta
	= \{ (t,p(x,\xi), G^{-t}(x,\xi)) \in \dot T^*(\R^{k+1} \times M^{k+1}) :
	(x,\xi) \in \dot T^* M^{k+1} \cap \esssupp A \},
\]
with principal symbol having half-density part
\[
	(2\pi)^{-(kn-n-k-1)/4} a(p(x,\xi)) |g(x_1)|^{-(k-1)/4}
	|dt \, dx_1 \, d\xi_2 \cdots d\xi_{k+1}|^{1/2},
\]
via the parametrization
\[
(x,\xi) = (x_1,x_1,\ldots,x_1,-\xi_2-\cdots-\xi_{k+1},\xi_2,\ldots,\xi_{k+1}).
\]
\end{proposition}

Recall from \eqref{strategy first composition} and \eqref{strategy second composition}
that $S : C^\infty(M^{k+1}) \to C^\infty(\R^{k+1})$ is the operator with distribution
kernel $(Q \circ A) \circ \delta_\Delta$. That is,
$S \in I^{(kn-n-k-1)/4}(\R^{k+1} \times M^{k+1}; \mathcal C_S')$, where
\begin{multline}\label{S canonical relation}
	\mathcal C_S
	= \{(t,p(x,\xi); G^{t}(x,\xi)) \in \dot T^*\R^{k+1} \times \dot T^*M^{k+1}: \\
	x = (x_1,\ldots,x_1) \in M^{k+1},\ 
	\xi = (-\xi_2-\cdots-\xi_{k+1},\xi_2,\ldots,\xi_{k+1}) \in \dot T^*_x M^{k+1}, \\
	(x,\xi) \in \esssupp A \}.
\end{multline}
It has the same principal symbol as in Proposition \ref{prop first composition}.
Note that $\mathcal C_S \subset \dot T^*\R^{k+1} \times \dot T^*M^{k+1}$ whenever
$\esssupp A$ is contained in the positive octant, for example if it is $k$-good or
$k$-bad. The wavefront set of $S \circ \delta_\Delta$ can now be computed using the
standard calculus \cite[Chapter 1]{MR451313}. In doing so, certain configurations
appear that we isolate using the following lemma from \cite{MR4376455}, which we
apply in the proof of Proposition \ref{prop first composition}.

Let $Y$ be a smooth, compact manifold without boundary with $\dim Y = n$. Fix a
conic Lagrangian submanifold $\Lambda \subset T^*Y$. Let
$p = (p_1,\ldots,p_{k+1})$ be positive-homogeneous first-order symbols satisfying
$\{p_i, p_j\} = 0$ for all $1 \le i,j \le k+1$. In particular, the Hamilton vector
fields $H_{p_i}$ and $H_{p_j}$ commute for each pair.

Consider the canonical relation
\[
	\mathcal C_p
	= \{(t,p(x,\xi), G^{-t}(x,\xi) ; x,\xi) : t \in \R^{k+1},\ (x,\xi) \in \dot T^*Y \}
	\subset \dot T^*(\R^{k+1} \times Y) \times \dot T^*Y,
\]
where
\[
	G^t = \exp(t_1 H_{p_1}) \circ \cdots \circ \exp(t_{k+1} H_{p_{k+1}})
	= \exp(t_1 H_{p_1} + \cdots + t_{k+1} H_{p_{k+1}}).
\]
In what follows, we fix a parametrization $\lambda \mapsto (x(\lambda),\xi(\lambda))$
of $\Lambda$ by $\lambda=(\lambda_1,\ldots,\lambda_n)\in\R^n$. The composition
\[
	\mathcal C_p \circ \Lambda
	= \{(t,p(x,\xi), G^{-t}(x,\xi)) : t \in \R^{k+1},\ (x,\xi) \in \Lambda \}
\]
then admits a parametrization by $(t,\lambda)$.

\begin{lemma}\label{symbol composition lem trans}
The composition of $\mathcal C_p$ and $\Lambda$ is transversal. Furthermore, fix
homogeneous half-densities
\[
	\sigma_{\mathcal C_p}(t,x,\xi)
	= a(t,x,\xi)\,|dt \, dx \, d\xi|^{1/2}
	\qquad \text{and} \qquad
	\sigma_{\Lambda}(\lambda) = b(\lambda)\,|d\lambda|^{1/2}
\]
on $\mathcal C_p$ and $\Lambda$, respectively. Then the composition
$\sigma_{\mathcal C_p} \circ \sigma_{\Lambda}$ is the half-density
\[
	a(t,x(\lambda),\xi(\lambda))\, b(\lambda)\, |dt \, d\lambda|^{1/2}
\]
via the parametrization of $\mathcal C_p \circ \Lambda$ by $(t,\lambda)$.
\end{lemma}

\begin{proof}[Proof of Proposition \ref{prop first composition}]
Let $Y=M^{k+1}$, so $\dim Y=(k+1)n$. We identify the Lagrangian submanifold
$\dot N^{*}\Delta$ from Proposition \ref{delta symbol} with the $\Lambda$ in
Lemma \ref{symbol composition lem trans}, and we take the parameters
$\lambda=(x_1,\xi_2,\ldots,\xi_{k+1})$.

We also take
\[
p(x,\xi)=(p(x_1,\xi_1),p(x_2,\xi_2),\ldots,p(x_{k+1},\xi_{k+1})).
\]
Under the action of $A$, we have $(x,\xi)\in\esssupp A$ only when $p(x,\xi)$ lies
in a $k$-good cone, so each component $p(x_j,\xi_j)$ is positive.

In this setting,
\[
\sigma_{\mathcal C_p}(t,x,\xi)
=(2\pi)^{(k+1)/4} a(p(x,\xi))\,|dt \, dx \, d\xi|^{1/2}
\]
and
\[
\sigma_{\Lambda}(\lambda)
=(2\pi)^{-(k-1)n/4}\frac{1}{|g(x_1)|^{(k-1)/4}}
\,|dx_1 \, d\xi_2 \, d\xi_3 \cdots d\xi_{k+1}|^{1/2}.
\]
The proposition now follows directly from Lemma \ref{symbol composition lem trans}.
\end{proof}

Next, we compute the symbolic data of $S\circ \delta_\Delta$ and use it to obtain
the asymptotic expansion of $\check\mu\,|dt|^{1/2}$.

\begin{proposition}\label{main composition prop}
Suppose the conical support of $a$ is $k$-good. The restriction of the
composition $S \circ \delta_{\Delta}$ to the open cube $(-\inj M, \inj M)^{k+1}$
is a Lagrangian distribution in the class
\[
	I^{kn - 3(k+1)/4}\bigl((-\inj M, \inj M)^{k+1}; \dot T_0^* \R^{k+1}\bigr)
\]
with principal symbol having half-density part
\[
	(2\pi)^{-kn + 3(k+1)/4}\, \vol M\, a(\tau)\, \vol F^{-1}(\tau)\, |d\tau|^{1/2}.
\]
\end{proposition}

The proof of Proposition \ref{main composition prop} uses the clean composition
calculus of Duistermaat and Guillemin \cite{MR405514}. We first record two lemmas.

Let $Y$, $\Lambda$, and $p$ be as in Lemma \ref{symbol composition lem trans}, and
define
\[
	\mathcal C_{\Lambda}
	= \{(t,p(x,\xi); G^t(x,\xi)) : t \in \R^{k+1},\ (x,\xi) \in \Lambda \}
	\subset \dot T^*\R^{k+1} \times \dot T^*Y.
\]
Assume that the Hamilton vector fields
\[
	H_p = (H_{p_1}, \ldots, H_{p_{k+1}})
\]
are linearly independent and span a subspace that intersects $T\Lambda$ trivially.

\begin{lemma}\label{symbol composition lem}
Let $\mathcal C_{\Lambda}$ and $\Lambda$ be as above. The following are true.
\begin{enumerate}
\item There is an isolated component of
\[
	\mathcal C_{\Lambda} \circ \Lambda
	= \{(t,p(x,\xi)) \in \dot T^* \R^{k+1} : (x,\xi) \in \Lambda,\ G^t(x,\xi) \in \Lambda \}
\]
contained in $\dot T_0^*\R^{k+1}$.
\item At this component, $\mathcal C_{\Lambda}$ and $\Lambda$ compose cleanly
with excess
\[
	e = \dim Y - (k+1).
\]
\item Suppose $\Lambda$ is parametrized by $\lambda = (\lambda_1,\ldots,\lambda_n) \in \R^n$,
and that we have homogeneous half-densities
\[
	\sigma_{\mathcal C_{\Lambda}} = b(t,\lambda)\, |dt \, d\lambda|^{1/2}
	\qquad \text{and} \qquad
	\sigma_{\Lambda} = a(\lambda)\, |d\lambda|^{1/2}
\]
on $\mathcal C_{\Lambda}$ and $\Lambda$, respectively. Then
\[
	\sigma_{\mathcal C_{\Lambda}} \circ \sigma_\Lambda
	= \left(\int_{p^{-1}(\tau)} b(0,\lambda)\, a(\lambda)\, d_p(\lambda) \right) |d\tau|^{1/2}
\]
at $(0,\tau) \in \dot T^*\R^{k+1}$, where $d_p$ denotes the Leray density on
$p^{-1}(\tau)$.
\end{enumerate}
\end{lemma}

\begin{lemma}\label{d}
Let $\mathcal C$ and $\Lambda$ compose cleanly with excess $e$. If $A$ is a
Fourier integral operator associated to $\mathcal C$ and $B$ is a Lagrangian
distribution associated to $\Lambda$, then $A \circ B$ is a Lagrangian
distribution associated to $\mathcal C \circ \Lambda$ with order
\[
	\ord(A \circ B) = \ord A + \ord B + \frac{e}{2}.
\]
Furthermore, if $\sigma_A$, $\sigma_B$, and $\sigma_{A \circ B}$ are the
half-density parts of the principal symbols of $A$, $B$, and $A \circ B$,
respectively, then
\[
	\sigma_{A \circ B} = (2\pi i)^{-e/2}\, \sigma_A \circ \sigma_B.
\]
\end{lemma}

Proofs of Lemma \ref{symbol composition lem} and Lemma \ref{d} can be found in
\cite{MR4376455} and \cite{MR405514}, respectively. We now prove
Proposition \ref{main composition prop}.

\begin{proof}[Proof of Proposition \ref{main composition prop}]
For $1 \le j \le k+1$, let $H_j$ denote the vector field on $T^*M^{k+1}$ whose
$j$th component is $H_p(x_1,\xi_j)$ and whose other components vanish, that is,
\[
	H_j = (0,\ldots,0,H_p(x_1,\xi_j),0,\ldots,0).
\]
These $H_j$ are linearly independent, and the subspace they span consists of
vectors of the form
\[
	(t_1' H_p(x_1,\xi_1),\, t_2' H_p(x_1,\xi_2),\, \ldots,\, t_{k+1}' H_p(x_1,\xi_{k+1})).
\]
We claim that this subspace intersects $T(\dot N^*\Delta)$ trivially. To see
this, suppose
\[
	(t_1' H_p(x_1,\xi_1),\, \ldots,\, t_{k+1}' H_p(x_1,\xi_{k+1})) \in T(\dot N^*\Delta).
\]
In geodesic normal coordinates about $x_1$, the Hamilton vector field has the
form
\[
	H_p(0,\xi) = \sum_{m=1}^n \frac{\xi_m}{|\xi|}\,\frac{\partial}{\partial x_m}.
\]
Therefore, the pushforward of this vector under $T^*M^{k+1} \to M^{k+1}$ equals
\[
	(t_1'\xi_1/|\xi_1|,\, \ldots,\, t_{k+1}'\xi_{k+1}/|\xi_{k+1}|).
\]
Since $T\Delta \subset TM^{k+1}$ consists of diagonal vectors, membership in
$T\Delta$ forces these components to agree. In particular, all $\xi_j$ must be
parallel. On $\dot N^*\Delta$ we also have $\xi_1 = -(\xi_2+\cdots+\xi_{k+1})$, so
parallelism implies the existence of a sign vector
$\varepsilon \in \{\pm 1\}^{k+1}$ with
\[
	\varepsilon_1|\xi_1| + \cdots + \varepsilon_{k+1}|\xi_{k+1}| = 0.
\]
But $(|\xi_1|,\ldots,|\xi_{k+1}|)$ lies in a $k$-good cone on $\esssupp A$, so
this contradicts the nondegeneracy condition \eqref{2.2}. Hence
$t_1'=\cdots=t_{k+1}'=0$, which proves the claim.

Part (1) of Lemma \ref{symbol composition lem} implies that the composition
$\mathcal C_{\Lambda} \circ \dot N^*\Delta$ has an isolated component in
$\dot T_0^*\R^{k+1}$. Moreover, if there exists $t=(t_1,\ldots,t_{k+1})\neq 0$
such that the points $\exp(t_1H_p)(x_1,\xi_1),\ldots,\exp(t_{k+1}H_p)(x_1,\xi_{k+1})$
project to the same base point in $M$, then at least one $|t_j|$ is at least
$\inj M$. Since we restrict to $(-\inj M,\inj M)^{k+1}$, we may therefore ignore
this case.

Part (2) of Lemma \ref{symbol composition lem} gives the excess of the clean
composition at the $t=0$ component,
\[
	e = (k+1)n - (k+1).
\]
Hence,
\begin{align*}
	\ord(S \circ \delta_\Delta)
	&= \ord S + \ord \delta_\Delta + \frac{e}{2} \\
	&= \frac{kn-n-k-1}{4} + \frac{(k-1)n}{4} + \frac{(k+1)n - (k+1)}{2} \\
	&= kn - \frac{3(k+1)}{4},
\end{align*}
as claimed.

By part (3) of Lemma \ref{symbol composition lem}, we have locally
\[
	\sigma_{S} \circ \sigma_{\delta_\Delta}
	= (2\pi)^{(-2kn+2n+k+1)/4}
	\left(\int_{p(x,\xi) = \tau} |g(x_1)|^{-(k-1)/2}\, a(p(x,\xi))\, d_p \right)
	|d\tau|^{1/2}.
\]
Here $(x,\xi)$ denotes the point
\[
(x_1,\ldots,x_1;\ -\xi_2-\cdots-\xi_{k+1},\ \xi_2,\ldots,\xi_{k+1}),
\]
and the integral is taken over the variables $(x_1,\xi_2,\ldots,\xi_{k+1})$
satisfying $p(x,\xi)=\tau$, with respect to the Leray density $d_p$. By a simple
change of variables, we may rewrite this integral as
\begin{equation}\label{almost there}
	a(\tau) \int_M |g(x_1)|^{1/2}
	\left(\idotsint |g(x_1)|^{-k/2}\, \delta(p(x,\xi) - \tau)\, d\xi_2 \cdots d\xi_{k+1}\right)
	dx_1.
\end{equation}
On $\dot N^*\Delta$ we have $\xi_1 = -(\xi_2+\cdots+\xi_{k+1})$, hence
\[
	p(x,\xi) = (|\xi_1|,|\xi_2|,\ldots,|\xi_{k+1}|)
	= (|\xi_2+\cdots+\xi_{k+1}|,|\xi_2|,\ldots,|\xi_{k+1}|).
\]
Thus, up to a permutation of the target coordinates, the constraint
$p(x,\xi)=\tau$ is equivalent to $F(\xi_2,\ldots,\xi_{k+1})=\tau$ with $F$ as in
\eqref{def F}. Correspondingly, the inner integral in \eqref{almost there} is
\[
	\idotsint \delta(F(\xi_2,\ldots,\xi_{k+1}) - \tau)\, d\xi_2 \cdots d\xi_{k+1}
	= \vol F^{-1}(\tau).
\]
Because the $\xi$-integration is taken with the canonical cotangent-fiber density
induced by a $g$-orthonormal coframe, any chart Jacobian factor such as
$|g(x_1)|^{-k/2}$ in \eqref{almost there} is a coordinate artifact and cancels.
Finally,
\[
	\int_M |g(x_1)|^{1/2}\, dx_1 = \vol M.
\]
This completes the proof of Proposition \ref{main composition prop}.
\end{proof}

So far, we have completed the study of these composition operators. Next, we
only need to perform some Fourier transforms to prove
Theorem \ref{main theorem k+1-good}.

\begin{proof}[Proof of Theorem \ref{main theorem k+1-good}]
For convenience, we test $S\circ\delta_{\Delta}$ against the oscillatory
half-density
\[
	\phi_\tau(t) = \check \rho(t)\, e^{-i\langle t, \tau \rangle}\, |dt|^{1/2},
\]
where $\rho$ is as in Theorem \ref{main theorem k+1-good}. On one hand,
\eqref{strategy main composition} yields
\[
	(S \circ \delta_\Delta, \phi_\tau)
	= (2\pi)^{k+1} \int_{\R^{k+1}} e^{-i\langle t, \tau \rangle}\,
	\check \mu_a(t)\, \check \rho(t)\, dt
	= \rho * \mu_a(\tau).
\]

On the other hand, by \cite[(25.1.13)]{MR1481433} and the discussion preceding
it, the product $\check\rho(t)\,(S \circ \delta_\Delta)$ can be written as
\[
	\check\rho(t)\,(S \circ \delta_\Delta)
	= (2\pi)^{-3(k+1)/4}
	\left(\int_{\R^{k+1}} e^{i\langle t,\tau \rangle}\, \nu(\tau)\, d\tau \right)
	|dt|^{1/2},
\]
where, by Proposition \ref{main composition prop},
\[
	\nu(\tau) \equiv \zeta\, \check \rho(0)\, (2\pi)^{-kn + 3(k+1)/4}\,
	\vol M\, a(\tau)\, \vol F^{-1}(\tau)
	\quad \bmod S^{kn-k-2}(\R^{k+1} \setminus 0).
\]
Here $S^{kn-k-2}(\R^{k+1} \setminus 0)$ denotes the symbol class of order
$kn-k-2$, and $\zeta$ is a complex unit coming from the Maslov factors. By
Fourier inversion,
\[
	(S \circ \delta_\Delta, \phi_\tau)
	= (2\pi)^{k+1-3(k+1)/4}\, \nu(\tau)
	= \zeta\, \check \rho(0)\, (2\pi)^{k+1-kn}\,
	\vol M\, a(\tau)\, \vol F^{-1}(\tau)
	+ O(|\tau|^{kn-k-2}).
\]

Choose $a \ge 0$ and $\rho$ nonnegative. Then $\rho * \mu_a(\tau)$ is real and
nonnegative, and the leading term on the right-hand side is real and positive.
Hence $\zeta=1$. Moreover, $\check \rho(0) = (2\pi)^{-(k+1)}$ since $\int \rho =
1$. Therefore,
\[
	\rho * \mu_a(\tau)
	= (2\pi)^{-kn}\, \vol M\, a(\tau)\, \vol F^{-1}(\tau) + O(|\tau|^{kn-k-2}).
\]
The theorem follows after selecting $a \equiv 1$ on the $k$-good cone $\Gamma$
and $a \equiv 0$ outside a $k$-good cone containing $\Gamma$ in its interior.
\end{proof}

\section{A flat-torus example}\label{sec:flat-torus-example}

The following example illustrates the sharpness of the cutoff in Theorem \ref{main theorem k+1-gon-bad}.
\begin{example}\label{torus example}
Let $\T^n = \R^n/2\pi \Z^n$ denote the flat torus. Its standard $L^2$-orthonormal eigenfunctions are the Fourier exponentials
\[
	\varphi_m(x) = (2\pi)^{-n/2} e^{i\langle x, m \rangle} \qquad \text{for $m \in \Z^n$,}
\]
with corresponding frequency $|m|$. For $m_1, m_2, \dots, m_{k+1} \in \Z^n$, the $(k+1)$-fold product satisfies
\begin{align*}
	\langle \varphi_{m_1} \varphi_{m_2} \cdots \varphi_{m_k}, \varphi_{m_{k+1}} \rangle
	&= (2\pi)^{-(k+1)n/2} \int_{\T^n} e^{i\langle x,\, m_1 + m_2 + \cdots + m_k - m_{k+1} \rangle} \, dx \\
	&=
	\begin{cases}
		(2\pi)^{-(k-1)n/2} & \text{if } m_1 + m_2 + \cdots + m_k - m_{k+1} = 0, \\
		0 & \text{otherwise.}
	\end{cases}
\end{align*}
Consequently, the sum \eqref{single sums} equals
\[
	(2\pi)^{-(k-1)n} \#\Bigl\{ (m_1,m_2,\dots,m_k) \in \Z^{kn} : |m_1| = a_1, \dots, |m_k|=a_k,\ |m_1 + m_2 + \cdots + m_k| = a_{k+1} \Bigr\}.
\]
In particular, \eqref{single sums} vanishes unless there exist vectors of lengths $a_1,\dots,a_{k+1}$ that form a closed $(k+1)$-gon (Figure \ref{2}). When $k=2$, the shape of a triangle with side lengths $a_1,a_2,a_3$ is rigid. For $k>2$, the shape of a $(k+1)$-gon with fixed side lengths need not be unique. The resulting polygons may be non-convex and may even self-intersect (Figures \ref{2} and \ref{3}).

This shows that the cutoff in Theorem \ref{main theorem k+1-gon-bad} is essentially sharp. Indeed, fix a nonzero lattice vector $v \in \Z^n$ and integers $r_1,\dots,r_k\ge 1$, and set $m_j=r_j v$ for $1\le j\le k$ and $m_{k+1}=(r_1+\cdots+r_k)v$. Then $|m_j|=r_j|v|$ and $|m_{k+1}|=(r_1+\cdots+r_k)|v|$, so $\lambda_{m_{k+1}}=\lambda_{m_1}+\cdots+\lambda_{m_k}$, while
\[
\bigl|\langle \varphi_{m_1} \varphi_{m_2} \cdots \varphi_{m_k}, \varphi_{m_{k+1}} \rangle\bigr| = (2\pi)^{-(k-1)n/2}.
\]
In particular, there is no decay at the boundary of the $(k+1)$-gon inequality, so the frequency threshold in Theorem \ref{main theorem k+1-gon-bad} is essentially sharp.
\end{example}

\begin{figure}
\includegraphics[width=0.6\textwidth]{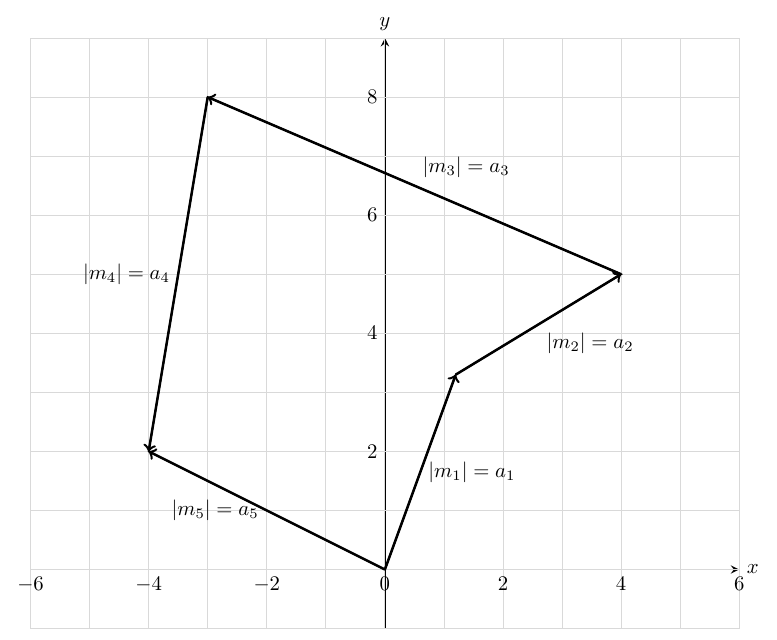}
\caption{A pentagon with side lengths $a_1, \ldots, a_5$.}
\label{2}
\includegraphics[width=0.6\textwidth]{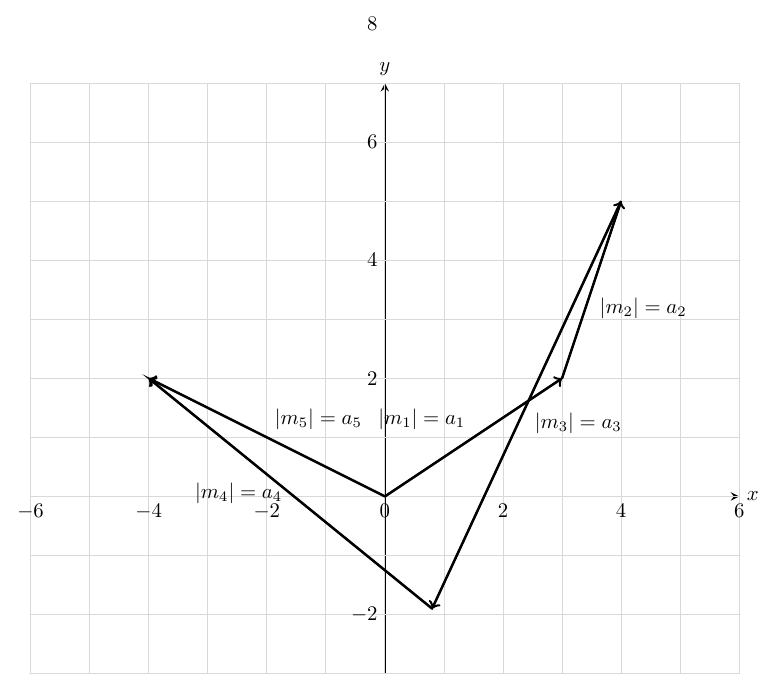}
\caption{A self-intersecting pentagon with the same side lengths.}
\label{3}
\end{figure}

\bibliographystyle{alpha}
\bibliography{references}

\end{document}